\documentclass[letterpaper, 10 pt, conference,twocolumn]{ieeeconf}  
\IEEEoverridecommandlockouts                              
\usepackage{graphicx} 
\usepackage[originalparameters]{ragged2e}
\usepackage{epsfig} 
\usepackage{tikz,pmat}
\usepackage{amsmath} 
\usepackage{amssymb}  
\usepackage{array}
\usetikzlibrary{arrows}
\usetikzlibrary{circuits.ee.IEC}
\usetikzlibrary{shapes.geometric}
\usetikzlibrary{shapes.multipart}
\usetikzlibrary{decorations.pathreplacing}
\usetikzlibrary{patterns}
\usepackage{framed}
\usepackage{lineno}
\usepackage{cite}
\usepackage{authblk}
\usetikzlibrary{decorations.pathmorphing}
\usepackage{algorithm}
\usepackage{algorithmic}
\usepackage{mathtools}
\usepackage{pgfplots} 
\usepackage{pgfgantt}
\usepackage{pmat}

\usepackage[utf8]{inputenc}
\usepackage[english]{babel}
\usepackage{multirow}
\usepackage{fancybox}
\usepackage{enumerate}
\newcommand{\R}{\mathbb{R}}
\newcommand{\U}{\mathbb{U}}
\newcommand{\N}{\mathbb{N}}
\newcommand{\W}{\mathbb{W}}
\newcommand{\X}{\mathbb{X}}

\DeclareMathOperator{\conv}{co}

\definecolor{ao}{rgb}{0.0, 0.5, 0.0}
\newtheorem{theorem}{Theorem}
\newtheorem{remark}{Remark}
\newtheorem{lemma}{Lemma}

\newtheorem{definition}{Definition}

\newtheorem{problem}{Problem}

\title{Sample Truncation for  Scenario Approach to Closed-loop Chance Constrained Trajectory Optimization for Linear Systems }

\author{Hossein Sartipizadeh,  Beh{\c{c}}et A{\c{c}}{\i}kme{\c{s}}e
	
	\thanks{H. Sartipizadeh is with University of Texas at Austin, TX, USA. Email: {\tt\small hsartipi@utexas.edu}, B. A{\c{c}}{\i}kme{\c{s}}e is with the Department of Aeronautics \& Astronautics in the University of Washington, Seattle, WA. Email: {\tt\small behcet@uw.edu}.}
}

\begin{document}

\maketitle
\thispagestyle{empty}
\pagestyle{empty}

\begin{abstract}

This paper studies   closed-loop chance constrained control problems with disturbance feedback (equivalently state feedback) where state and input vectors must remain in a prescribed  polytopic safe region with a predefined confidence level. We propose to use a {\em scenario approach}  where   the uncertainty is replaced with a set of random samples (scenarios).
Though a standard form of  scenario approach is applicable in principle, it typically requires a large number of samples to ensure the required confidence levels.    To resolve this drawback, we propose a method to reduce the computational complexity by {\em eliminating} the redundant samples and, more importantly, by  {\em truncating} the less informative samples. Unlike the prior methods that start from the full sample set and remove the less informative samples at each step, we sort the samples in a descending order by first finding the most dominant ones. In this process  the importance of each sample is measured via a proper mapping. 
Then  the most dominant samples can be selected based on the allowable  computational complexity and the rest of the samples are truncated offline. 
The truncation error is later compensated for by adjusting the safe regions via properly designed   buffers, whose sizes   are functions of the feedback gain and the truncation error.

\end{abstract}


\section{Introduction} \label{sec:Intro}

Computing solutions of finite horizon optimal control problems (FHOP) is a key capability  in controlling  dynamical systems with input and state constraints  \cite{pontryagin1987mathematical}. 
Given the initial state  of the system in a FHOP, the future state trajectory can be obtained for any input by propagating the states through an  explicit model.  The propagated trajectories that meet the constraints and minimize the objective cost are then selected as optimal input and state trajectories. 
In practice, however, uncertainties may impact the accuracy of the explicit model  and consequently the predicted state trajectory. A robust approach leads to a solution that satisfies the constraints for every uncertainty realizations \cite{campo1987robust,ben1998robust,bemporad1999robust}. Since uncertainties are typically of the stochastic nature, they can be large enough to make the robust control problem infeasible, i.e.,   a  solution satisfying the  constraints under all possible realizations of the uncertainty can not be found. Hence, this paper adopts a stochastic problem formulation,   also referred to as  \textit{Chance Constrained Trajectory Optimization} (CCTO), that relaxes  the control problem to  a probabilistic one where the constraints are to be met by a prescribed  confidence level $1-\delta$,  $\delta \in (0,1)$  \cite{charnes1959chance}.  

Solving CCTO, that is widely investigated  in  Stochastic Model Predictive Control (SMPC) context \cite{schwarm1999chance,li2000robust,li2002probabilistically} (earlier work)  and \cite{mesbah2016stochastic} (contains a comprehensive review), is typically challenging as it requires  calculating multi-dimensional integrals, which can result in a non-convex optimization problem \cite{nemirovski2006convex,blackmore2009convex}. Even   CCTO problems with convex constraints  are often computationally intractable \cite{nemirovski2006convex}. When the uncertainty distribution is assumed to be  known and there is additional structure in the problem (e.g., temporal spatial independence of uncertainty,  affine dependence of constraints on the uncertainty), the original CCTO can be analytically approximated by a conservative but tractable problem (e.g.\cite{ben2000robust,nemirovski2003tractable,bertsimas2004price,blackmore2009convex,rajendran2019stochastic,ansaripour2016some}). 
Nevertheless, existing approximation-based approaches either require restrictive assumptions to hold true, which may  limit their applicability.  

To tackle this issue, \emph{scenario approach} suggests to replace the uncertainty  with a finite number (say $N$) of so-called scenarios through a random sampling, and then find the robust solution for the sampled uncertainty set \cite{calafiore2006scenario,calafiore2013stochastic}. A bound on the sufficient number of samples is given in \cite{calafiore2006scenario,calafiore2010random} based on the required confidence level.  Since the samples are taken randomly, there is a risk of failure  in achieving the desired confidence level unlike the analytical methods. 
The main advantages of scenario approach are its generality and tractability: it converts the original stochastic problem to a convex problem regardless of the probability distribution of the uncertainty. That is, there is no need  to know what the distribution is; all we need is to be able to sample from the distribution. However, to achieve a reasonable risk of failure, a large number of samples is typically needed. 
The scenario approach may consequently result in a computationally expensive problem since the constraints must be checked at each sample.
Although some studies offer fewer number of samples (for instance, \cite{lorenzen2017stochastic,zhang2015convex, zhang2015onthesamplesize}) as compared to the original scenario approach by either presenting a tighter bound or discarding the redundant scenarios (e.g. \cite{calafiore2010random,campi2011sampling,lorenzen2017stochastic}), a fairly large number of samples are still needed to ensure the desired specification.

We have recently introduced an approximate convex hull-based framework to reduce the computational complexity of open-loop CCTO by truncating the samples without losing the desired confidence level \cite{sartipizadeh2016uncertainty,satipizadeh2018Automatica,sartipizadeh2018CCTO_OL}. 
Earlier scenario-based methods discard samples, hence reduce  the number of  samples, while increasing  the risk of failure (which is denoted by $\beta$ in the current paper).  On the contrary, our truncation approach preserves prescribed levels of risk of failure ($\beta$) and the confidence in constraint satisfaction (denoted by $\delta$) by buffering the feasible constraint sets to account for the  truncation error. 
To this end, we inner approximate the state uncertainty region of scenario approach using a subset of $\hat{N}$ scenarios, referred to as \textit{truncated sample set}, where  $\hat{N}$ can be chosen significantly smaller than $N$ (sufficient number of samples needed for scenario approach).  CCTO problem is then solved by checking the constraints only over the truncated sample set while the truncation error is compensated for by adjusting the safe region using a proper \textit{buffer}. In our proposed method, the truncated sample set is initiated from  an empty set as opposed to discarding samples that remove one sample at a time from the full sample set. Since the truncated sample set gives the best approximation of the convex hull of  original sample set  (with $N$ samples)     with an approximation  error of $\epsilon$, this method is
referred to as $\epsilon$-approximate convex hull. We use \textit{truncation} and \textit{approximation} interchangeably in this paper. 

This concept is illustrated in Figure~\ref{fig:Intro}. Assume that state $x$ at time $t$ is to remain in the safe region $\X$ with some confidence level. Let $N$ be the sufficient number of samples proposed by scenario approach for the desired $\delta$ and let $\mathcal{X}_{t}=\conv \{x_{t}^{(1)},\cdots,x_{t}^{(N)} \}$ be the convex hull of state trajectories  due to  the $N$ scenarios.
$\hat{\mathcal{X}}_{t}=\conv \{x_{t}^{(1)},\cdots,x_{t}^{(\hat{N})} \}$ is  an inner approximation of $\mathcal{X}_{t}$ obtained by $\hat{N}$ samples after truncation. $\mathcal{X}$ and $\hat{\mathcal{X}}$ are also referred to as \textit{original} and \textit{truncated uncertainty envelopes}, respectively.  The \textit{buffered constraint set} at time $t$ is shown as $\hat{\X}_{t}$ and is calculated such that the original  uncertainty envelope at time $t$ remains in the safe region $\X$ when  the truncated uncertainty envelope $\hat{\mathcal{X}}_{t}$ is kept in  $\hat{\X}_{t}\subseteq \X$. Consequently, we will impose that  $\hat{\mathcal{X}}_{t} \subseteq \hat{\X}_{t}$  to ensure the specified levels of confidence and risk of failure for  constraint satisfaction.

 In open-loop scheme, sample truncation and buffer computation are performed once and offline \cite{sartipizadeh2018CCTO_OL}. In addition, since the inputs are deterministic, buffers only need to be calculated for the state constraints.  However, closed-loop scheme with state feedback is  preferred due to the benefits of using the potential knowledge of future states as they become available. When the state feedback gain is itself a decision variable, the closed-loop CCTO becomes  a non-convex problem. To resolve this issue and set up a convex optimization problem, disturbance feedback can be equivalently used \cite{van2002conic,lofberg2003approximations,goulart2006optimization}, which is the approach adopted in this paper.    
 
  In this paper, we extend our sample truncation framework to the closed-loop  CCTO with disturbance feedback, where the feedback gain is also an optimization variable and can  be selected from a set of stabilizing gains with a common Lyapunov function.  In this case, unlike the open-loop scheme, the input is also stochastic and needs buffering to compensate for the truncation error. Furthermore, buffers are dynamic variables and  are  functions of the feedback gain.

\begin{figure}[h!]
	\begin{center}
		\scalebox{.85}{
		\begin{tikzpicture}[thick,scale=1]
		\coordinate (a1) at (0,0);
		\coordinate (b1) at (0.1,0.1);
		\coordinate (c1) at (0.15,0.15);
		\coordinate (d1) at (0.3,0.3);
		\coordinate (e1) at (1.5,1);
		
		%
		%
		\draw (a1) -- (0:5cm) node[below] {$\X$} -- ++ (40:2cm)  -- ++ (140:4cm) -- ++ (190:3cm) -- cycle ;
		\draw [thick,dashed, black!70!] (c1) -- ++ (0:4.45cm) node[above] {$\hat{\X}_{t}$} -- ++ (40:1.93cm) -- ++ (140:3.45cm) -- ++ (190:2.87cm) -- cycle ;
		
		%
		%
		\draw [thick,black!100!,fill=black!10!] (e1) -- ++ (-10:1cm) coordinate(e2){} -- ++ (0:1cm) coordinate(e3){} -- ++ (40:1cm)coordinate(e4){} -- ++ (120:1cm) coordinate(e5){} -- ++ (150:.5cm) coordinate(e6){} -- ++ (180:.8cm) coordinate(e7) -- ++ (210:.5cm) coordinate(e8){} -- cycle ;
		
		
		\draw [thick, dashed,black!70!, fill=black!20!] (e1) --(e3) --(e4)-- (e6) -- (e8) --cycle;	
					
		\draw[<-,>=stealth] (6.0,1.75) -- ++(0.5,.7) node[above] {original constraints};
		\draw[<-,>=stealth,dashed] (4.75,2.5) -- ++(.7,1) node[above] {buffered constraints};
				
		\node at (e7) [above] {$ \mathcal{X}_{t}$};
		\node at (3,1.5) [above] {$ \hat{\mathcal{X}}_{t}$};
		
		\end{tikzpicture}
	}
	\end{center}
\vspace{-5mm}
	\caption{Uncertainty characterization for polytopic constraints using an approximate convex hull. $\mathcal{X}_{t}$, {\em uncertainty envelope},  denoted by grey represents the convex hull of  states at time $t$ obtained from $N$ samples of the disturbance signal. $\hat{\mathcal{X}_{t}}$, shown by dashed region, is its approximation with $\hat{N}$ extreme points, {\em truncated uncertainty envelope}.  keeping $\hat{\mathcal{X}}_{t}$ in $\hat{\X}_{t}$ at time $t$ guarantees that $\mathcal{X}_{t}$ remains in $\X$.} 
	\label{fig:Intro}
\end{figure}
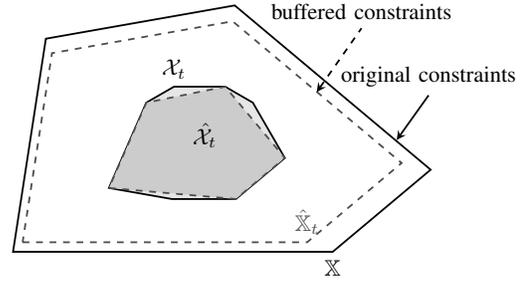


\textbf{Notation:}
$\R$ and $\N$ are sets of real and natural numbers, with $\R^{n}$ a length $n$ vector of real numbers and $\N_{\leq n}$ the set of natural numbers up to $n$.  For $x\in\R^{n}$, $x^{T}$ denotes the transpose of $x$ and $x_i$ is the $i$th element of $x$.  Also we define $\mathbb{C}=\{x\in\R^{n}| x_i\geq0, \sum_{i}x_{i}=1\}$ as the set of convex coefficients. Given a set $\W$ with elements $w$ and function $f(w)$, let $f(\W)$ denote the set $\{f(w) | w \in \W\}$. $I_{n}\in\R^{n\times n}$ is the identity matrix. \textbf{1} is a vector of ones of compatible dimension.  Given a set $\W$, $\mathcal{W}=\conv{\W}$ is the convex hull of $\W$. $\W_{1} \backslash\W_{2}$ is the set obtained by removing the elements in $\W_{2}$ from $\W_{1}$. Let  $M_{1}$ to $M_{m}$ be $m$ matrices from the same size. $vec(M_{1},\cdots,M_{m})$ shows the vectorized form of the concatenated matrix $M=[M_{1},\cdots, M_{m}]$.


\section{Problem Description} \label{sec:ProblemDescription}

We consider the following discrete-time LTI system, 
\begin{equation} \label{eq:sys_dynamics}
x_{t+1}=Ax_{t}+B_{u}u_{t}+ B_{w}w_{t}\\
\end{equation} 
where $x_{t}\in\R^{n_x}$ and $u_{t}\in \R^{n_{u}}$, are the system state and input vectors at sample time $t$, respectively, and  $ w_{t} \in \W \subseteq \R^{n_w}$ where $\W$ represents an unbounded disturbance set  which its probability distribution is not necessarily known, but it can be sampled from. Given the system dynamics \eqref{eq:sys_dynamics},  the design goal in a CCTO problem is to determine control inputs belonging to set $\U$ over a $p$-length future horizon such that the resulting state trajectory remains in a desired set $\X$ with a prescribed confidence. In this paper, $\X$ and $\U$ are assumed to be polytopes,
 \begin{subequations}
 \begin{align}
 \X&=\{x|f^{x}x\leq \textbf{1}\}, \\
 \U&=\{u|f^{u}u\leq \textbf{1}\}, 
 \end{align}
 \end{subequations}
where $f^{x}\in\R^{n_{c_x}\times n_x}$ and $f^{u}\in\R^{n_{c_u}\times n_u}$. 
Let $p$ be a positive integer and without loss of generality assign $t=0$ to the current, decision, time and $t>0$ are time epochs in the future. By stacking input, output, and state as,
\begin{equation}
X=\begin{bmatrix} x_{1}  \\ \vdots \\x_{p} \end{bmatrix},
U=\begin{bmatrix} u_{0}  \\ \vdots \\u_{p-1} \end{bmatrix}, 
W=\begin{bmatrix} w_{0}  \\ \vdots \\w_{p-1} \end{bmatrix},
\end{equation}
we define the stacked system as
\begin{equation} \label{eq:sys_stacked}
X= G_{x}x_{0} + G_{u} U + G_{w}W,
\end{equation}
where $x_{0}$ is the initial state and,
\begin{equation}
\begin{aligned}
G_{x}&=\begin{bmatrix} A \\ A^{2} \\ \vdots \\A^{p} \end{bmatrix},
G_{w}&=\begin{bmatrix} B_{w}    \\ AB_{w} & \hspace{8mm} \mbox{\Large \(0 \)}  \\ \vdots &\ddots  &\\A^{p-1}B_{w} &\cdots & B_{w} \end{bmatrix}.\\
\end{aligned}
\end{equation}
Furthermore, $G_{u}$ is defined similar to $G_{w}$ by replacing $B_{w}$ with $B_{u}$. 

\subsection{State Feedback vs. Disturbance Feedback}
In open-loop scheme, at any time $t\in\N_{\leq p}$, the input $u_{t}$ is computed only  based on the initial state $x_0$. When a full state feedback is available, one can also use the knowledge of states $x_{1}$ to $x_{t-1}$ to calculate $u_{t}$ 
at the future times $t\in\N_{\leq p}$. For instance, the input policy in a typical tracking error feedback with respect to the nominal state trajectory is defined as $u_{t}=L_{t}e_{t}+\hat{u}_{t}$ with $e_{t}=x_{t}-\hat{x}_{t}$ where $\hat{x}$ and $\hat{u}$  denote the nominal state and input vectors, and are calculated  via  the  nominal system $\hat{x}_{t+1}=A\hat{x}_{t}+B_{u}\hat{u}_{t}$. 
This feedback policy results in the error dynamics $e_{t+1}=(A+B_{u}L_{t})e_{t}+B_{w}w_{t}$.  This closed-loop input policy can be written as $u_{t}=L_{t}x_{t}+g_{t}$ with $g_{t}=L_{t}\hat{x}_{t}+\hat{u}_{t}$, which is also captured within the general from, 
%
%
\begin{equation} \label{eq:inputpolicy_sf}
u_{i}=\sum_{j=0}^{i} L_{i,j}x_{j}+g_{i}, \hspace{15mm} i=0,\cdots,p-1.
\end{equation}
 In \eqref{eq:inputpolicy_sf}, the input at time $t$ is parametrized as a time varying affine function of states  up to time $t$ while $L_{i,j}$ and $g_{i}$ are to be calculated online. It can be easily investigated that the CCTO problem for system \eqref{eq:sys_dynamics} with the input policy \eqref{eq:inputpolicy_sf} is non-convex \cite{lofberg2003approximations,goulart2006optimization}. 

An alternative to state feedback policy \eqref{eq:inputpolicy_sf} is disturbance feedback as suggested in  \cite{van2002conic,lofberg2003approximations,goulart2006optimization}. Note that for any $j\leq i$ in \eqref{eq:inputpolicy_sf} $x_{j}=Ax_{j-1}+B_{u}u_{j-1}+B_{w}w_{j-1}$,  the input can be alternatively parametrized as  an  affine function of previous disturbances, 
\begin{equation} \label{eq:inputpolicy_df}
u_{i}=\sum_{j=0}^{i-1} K_{i,j}w_{j}+v_{i}, \hspace{15mm} i=0,\cdots,p-1.
\end{equation}
This can be shown in the stacked form $U=KW+V$ with,
\begin{equation} \label{eq:KV}
V=\begin{bmatrix} v_{0}  \\ \vdots \\v_{p-1} \end{bmatrix}, \ \ 
K=\begin{bmatrix} 0   &\cdots &\cdots &0 \\ K_{1,0} &0 &\cdots &0  \\ \vdots &\ddots  &\ddots  &\vdots\\K_{p-1,0} &\cdots &K_{p-1,p-2}& 0 \end{bmatrix}.
\end{equation}

The input policy \eqref{eq:inputpolicy_df} is easy to implement since it has a causal structure, i.e., the input at time $t$ only depends on disturbances up to time $t-1$. Note that the disturbance sequences can be calculated as the difference between actual and nominal states. In addition, since the disturbance is not a decision variable as opposed to state, policy \eqref{eq:inputpolicy_df} for a convex $\W$ set results in a convex optimization problem. Furthermore, disturbance feedback is in fact equivalent to the state feedback (see Theorem~9 in \cite{goulart2006optimization}), i.e., given $K$ and $V$, one can easily derive equivalent gains for policy \eqref{eq:inputpolicy_sf} (transformation mapping is given in \cite{goulart2006optimization}, Section~5). 

\subsection{Chance Constrained Trajectory Optimization} \label{sec:ProblemDes_StochasticMPC}


Here we present  a standard chance constrained control problem formulation that ensures state and control constraint satisfaction  by a predefined  probability $1-\delta$, also called \textit{confidence level}, for some $\delta\in[0,1)$ in exchange for reducing conservatism and achieving feasibility. Therefore, the original closed-loop CCTO problem with disturbance feedback $U=KW+V$ is defined as

\begin{problem} \label{prb:OriginalLifted}  Original CCTO
\begin{equation}
\begin{aligned}
\begin{aligned}
\displaystyle\min_{K,V} ~ &\quad \mathbb{E}(J(X,U)) \\
\mbox{s.t. } & \quad
X= G_{x}x_{0} + G_{u} V +(G_{u} K+ G_{w}) W \\
& \quad \mathbb{P}(F^{x}X\leq \textbf{1}) \geq 1-\delta \hspace{20mm} W\in\W^{p} \\
& \quad \mathbb{P}(F^{u}U\leq \textbf{1}) \geq 1-\delta 
\end{aligned} 
\end{aligned}
\end{equation}
\end{problem}
\vspace{2mm}
where $\mathbb{P}$ denotes the probability measure, $F^{x}=I_{p}\otimes f^{x}$, $F^{u}=I_{p}\otimes f^{u}$, and $\W^{p}=\W\times\cdots \times \W$ ($p$ times). 


\subsection{Sampling based approach} \label{sec:Scenario}

Problem~\ref{prb:OriginalLifted} is quite often nonconvex since a multidimensional integration is typically required to compute the probability measure. In order to convert  Problem~\ref{prb:OriginalLifted} to a convex problem, scenario approach suggests to check the constraints over a finite set of random disturbances (sampled from the distribution) $\W_{N}=\{W^{(1)},\cdots,W^{(N)}\}$, instead of  $\W^{p}$.
\begin{problem} \label{prb:Scenario}  CCTO via scenario approach
	\begin{align} \label{eq:Scenario}
	\begin{aligned}
	\displaystyle\min_{K,V} ~ &\quad E(J(X,U)) \\
	\mbox{s.t. } & \hspace{1mm}
	X^{(i)}= G_{x}x_{0} + G_{u} V +(G_{u} K+ G_{w})W^{(i)} \\
	& \hspace{1mm}  F^{x}X^{(i)}\leq 1  \hspace{26mm}  i=1,\cdots,N\\
	& \hspace{1mm}  F^{u}(KW^{(i)}+V)\leq 1  \\
	\end{aligned} 
	\end{align}
	\vspace{-3mm}	
\end{problem}
Let $K_{N}^{*}, V_{N}^{*}$ be the optimal solutions of Problem~\ref{prb:Scenario}. If the number of random disturbances is sufficient, any feasible $K_{N}^{*}, V_{N}^{*}$ will be also feasible for Problem~\ref{prb:OriginalLifted} with a probability (referred to as risk of failure) $1-\beta$. 
Specifically, for a new random disturbance $W \in \W^{p}$ and consequently for the resulting optimal input $U_{N}^{*}=K_{N}^{*}W+ V_{N}^{*}$, 
\begin{equation} \label{eq:ScenarioSpecification}
\begin{aligned}
 &\mathbb{P}(\mathbb{P}(X(U_{N}^{*})\in \X) \ngtr 1-\delta)\leq \beta, \\
 &\mathbb{P}(\mathbb{P}(U_{N}^{*}\in \U) \ngtr 1-\delta)\leq \beta. 
\end{aligned}
\end{equation}
 The required number of samples for guaranteeing this requirement  differs based on the problem, but in general case the following bound on $N$ is provided by the scenario approach \cite{calafiore2006scenario}
\begin{equation} \label{eq:scenarioN}
N \geq \frac{2}{\delta}\ln \frac{1}{\beta} +2n_{\theta}+\frac{2n_{\theta}}{\delta}\ln\frac{2}{\delta},
\end{equation} 	
where $n_{\theta}$ is the number of decision variables.
\begin{remark}
	For a quadratic cost function, Problem~\ref{prb:Scenario} is a QP since the disturbance has been characterized by a polytope with at most $N$ extreme points. 
\end{remark}

The computational complexity of Problem~\ref{prb:Scenario} increases with the number of samples since each sample imposes a new set of constraints. Typically a large number of samples is required  for a reasonably prescribed confidence level and risk of failure. The original scenario approach proposes sample reduction using a greedy method by removing one sample at a time. In detail, at step $k$, $N-k+1$ QP problems with $N-k$ scenarios are solved to remove the less dominant scenario (see \cite{calafiore2010random} for more detail). Though it is a systematic way to reduce number of samples, this approach typically demands a high computational complexity due to the need of solving a large number of QP problems over large sets of samples.  


\subsection{CCTO using Approximate Convex Hull}  \label{sec:CCTO_OL}

In order to decrease the online computational complexity, we have recently suggested a greedy truncation method to select the best $\hat{N} \ll N$ samples and discard the rest \cite{sartipizadeh2018CCTO_OL}: 
 \begin{equation}
\begin{aligned}
\W_{N}=\{\overbrace{W^{(1)}, \cdots,W^{(\hat{N})}}^\text{selected- $\W_{\hat{N}}$},
\overbrace{W^{(\hat{N}+1)},\cdots,W^{(E)}}^\text{truncated},\\
\underbrace{W^{(E+1)},\cdots, W^{(N)}}_\text{removed due to redundancy} \}.
\end{aligned}
\vspace{-1mm}
\end{equation}
The main idea behind this method is to approximate the original uncertainty envelope by a polytope with $\hat{N}$ extreme points and then account for the approximation error by adjusting the constraints using proper buffers. As a consequence, the sample truncation problem is simplified to a geometrical problem of computing the  best approximate convex hull of the original uncertainty envelope. The truncation is based on using a mapping $S(W)$ that reflects the uncertain parts of state equation~\eqref{eq:sys_stacked}. 

\begin{definition}
	 Let $S(W)\in \R^{n}$ be any mapping of $W$ and for arbitrary sets $\W_{N_1}=\{W_{N_1}^{(1)},\cdots,W_{N_1}^{(N_1)}\}$ and $\W_{N_2}=\{W_{N_2}^{(1)},\cdots,W_{N_2}^{(N_2)}\}$ where $\W_{N_1}\subseteq \W_{N_2}$ let
	%
	%
	\begin{equation} \label{eq:ShatS}
	\begin{aligned}
	\mathcal{S}_{1}=\conv S(\W_{N_1})=\conv\{S(W_{N_1}^{(1)}),\cdots,S(W_{N_1}^{(N_1)})\},	\\
	\mathcal{S}_{2}=\conv S(\W_{N_2})=\conv\{S(W_{N_2}^{(1)}),\cdots,S(W_{N_2}^{(N_2)})\}.
	\end{aligned}
	\end{equation} 
	Then the Hausdorff distance from $\mathcal{S}_{1}$ to  $\mathcal{S}_{2}$ is given by \cite{sartipizadeh2018CCTO_OL}
	\begin{equation} \label{eq:dH}
	\resizebox{0.43\textwidth}{!}{$
		d_{H}(\mathcal{S}_{1},\mathcal{S}_{2})= \displaystyle \max_{W\in\W_{N_{2}}} \min_{\alpha\in \mathbb{C}} \left\|S(W)-\sum_{i=1}^{N_{1}}\alpha_{i} S(W_{N_1}^{(i)})\right\|_{\infty}.
		$}
	\vspace{2mm}
	\end{equation}
\end{definition}
	Due to infinity norm definition and since $S(W)$ is a vector, \eqref{eq:dH} is simplified to a search over max and min of some scalars. Here, let $[S(W)]_{k}$ indicate the $k^{th}$ row of $S(W)$ and for $k\in \N_{\leq n}$  define
	\begin{equation} \label{eq:EpsilonXF_kl}
	\begin{aligned} 
	[\varepsilon]_{k} =\max\{ &\max_{W\in\W_{N}}[S(W)]_{k} - \max_{\hat{W}\in\W_{\hat{N}}}[S(\hat{W})]_{k} , \\
	&\min_{\hat{W}\in\W_{\hat{N}}}[S(\hat{W})]_{k}-\min_{W\in\W_{N}}[S(W)]_{k}\} .
	\end{aligned}
	\end{equation}	
Hence \eqref{eq:dH}  simplifies to 
\begin{equation} \label{eq:d_H_inf}
	d_{H}(\mathcal{S}_{1},\mathcal{S}_{2})=\|\varepsilon\|_{\infty}.
\end{equation}

 \begin{definition} [$\epsilon$-ACH] \label{def:ACH}
	$\hat{\mathcal{S}}$ is the best $\epsilon$-approximate convex hull of $\mathcal{S}$ iff $\hat{\mathcal{S}}$ is the smallest subset of $\mathcal{S}$, with fewest number of elements, so that $d_{H}(\hat{\mathcal{S}}, \mathcal{S})\leq \epsilon$.
\end{definition}

 When there is no feedback ($K=0$), the input constraints are deterministic and can be satisfied if a feasible solution of the state constrained optimization problem exists. Define $S(W)=F^{x}G_{w}W$ and let $\mathcal{S}=\conv S(\W_{N})$. Find the best $\hat{N}$ samples such that $d_{H}(\hat{\mathcal{S}}, \mathcal{S})\leq \epsilon$ where $\hat{\mathcal{S}}=\conv S(\W_{\hat{N}})$. $\W_{\hat{N}}$ can  simply be obtained using a greedy approach, which is initiated by a random extreme point of $S(\W_{N})$ and continued by recursively adding the sample that minimizes the Hausdorff distance (see \cite{sartipizadeh2018CCTO_OL} for details). Given $\W_{\hat{N}}$, the following problem with only $\hat{N}$ set of constraints and the state buffer $\varepsilon^{x}$ is suggested as a replacement for Problem~\ref{prb:Scenario}.     

	\vspace{2mm}
\begin{problem} \label{prb:ACH_OL}  Open-loop CCTO via $\epsilon$-ACH
	\begin{align} \label{eq:ACH_OL}
	\begin{aligned}
	\displaystyle\min_{V} ~ &\quad E(J(X,V)) \\
	\mbox{s.t. } & \hspace{1mm}
	X^{(i)}= G_{x}x_{0} + G_{u} V + G_{w}W^{(i)} \\
	& \hspace{1mm}  F^{x}X^{(i)}\leq \textbf{1}-\varepsilon^{x}  \hspace{20mm}  i=1,\cdots,\hat{N}\\
	& \hspace{1mm}  F^{u}V\leq \textbf{1 } \\
	\end{aligned} 
	\end{align}
	\vspace{-3mm}	
\end{problem}
$\varepsilon^{x}\in\R^{n_{c_x}p}$ is simply calculated through \eqref{eq:EpsilonXF_kl} where each element of $\varepsilon^{x}$ indicates the buffer on a specific constraint at some specific time. 

\begin{remark}
	Although buffering may reduce the feasible set of solutions, it preserves the original confidence level and failure risk. Let $V^{*}_{\hat{N}}$ and $V^{*}_{N}$ be the optimal solutions of Problem~\ref{prb:ACH_OL} and Problem~\ref{prb:Scenario} when $K=0$, respectively. For any $W\in\W^{p}$ the following conditions  hold \cite{sartipizadeh2018CCTO_OL}:
	\begin{equation}
	\begin{aligned}
		\mathbb{P}(\mathbb{P}(X(V_{\hat{N}}^{*})\in \X^{p}) \ngtr 1-\delta)\\
        \leq  \mathbb{P}(\mathbb{P}(X(V_{N}^{*})\in \X^{p}) \ngtr 1-\delta)\leq \beta, 
	\end{aligned}
	\end{equation}
\end{remark}

It is noted that the truncated sample set and buffers are computed once offline, although they may be updated online for time-varying constraints. Truncating more samples speeds up the online computation by effectively decreasing the number of constraints. However, a bigger buffer may be imposed to the constraints, which may decrease the size of the feasible region (and hence increase the conservatism of the solution). This trade-off can be selected by user and based on the application.


\section{ACH-CCTO with  Disturbance Feedback}

With the inclusion of feedback control, the input is also a stochastic variable and needs to remain in its safe region by a prescribed confidence level. Consequently, sample truncation also impacts the input constraint satisfaction. Therefore, similar to state constraints, input constraints need to be adjusted using proper buffers to preserve the probabilistic requirements. In addition, since the uncertainty envelope is a function of the feedback gain $K$,  input and state buffers must be also synthesized as functions of $K$.

In this section, we extend our sample truncation method for CCTO problem with disturbance feedback (Problem~\ref{prb:OriginalLifted}) while the feedback gain $K$ is defined as in \eqref{eq:KV}.  
 In Section~\ref{sec:TruncationMapping}, we present a proper truncation mapping $S(W)$ to extract $\hat{N}$ samples. Given the truncation mapping we define the state and input buffers in Section~\ref{sec:ComputingBuffers} and finally set up the truncated control problem as an alternative to the original CCTO problem.


\subsection{Sample Truncation} \label{sec:TruncationMapping}
Similar to the open-loop case, we define the truncation mapping $S(W)$ using the uncertain terms of state and input trajectories to capture the uncertainty in the direction of  normal vectors that describe the polytopes for the state and control constraints.  

The contribution of uncertainties  to state constraints are captured by the following mappings  $F^xG_{u}KW$ and $F^{x}G_{w}W$ while $F^{u}KW$ captures  the contribution of uncertainty to input constraints. Note that $K$ is a decision variable and cannot be used in defining the truncation mapping, and consequently sample truncation, since the truncation is performed offline and  $K$ is not given a priori.  On the other hand,   $F^xG_{u}KW$ and  $F^{u}KW$ contain $K$. To resolve this problem, we suggest to reorder these two mappings  by  using Lemma~\ref{lem:matrixSwitch}. This reordering  gives us the flexibility to truncate the samples and compute preliminary buffers based on the parts of these mappings that do not depend on $K$, offline, and then update the buffers according to $K$, online. 

\begin{lemma}[reordering matrix multiplication] \label{lem:matrixSwitch}
	Let	$A= \begin{pmat} [{|||}]
	a_1 & a_2 & \cdots &a_{z} \cr	\end{pmat} \in\R^{n\times z}$ and $B= \begin{pmat} [{|||}]
	b_1^{T} & b_2^{T} & \cdots &b_{z}^{T} \cr	\end{pmat}^{T}\in\R^{z\times m}$ be two matrices with $a_{i}$ and $b_{j}$ denoting the $i^{th}$ column of $A$ and $j^{th}$ row of $B$, respectively. It is proven that $AB=\bar{B}\bar{A}$ for
\begin{equation} \label{eq:matrixSwitch}
\begin{aligned}
\bar{B}&=\begin{bmatrix}b_1 &b_2& \cdots&b_z\end{bmatrix}\otimes I_n \\
\bar{A}&=\begin{bmatrix} 
(I_{m}\otimes a_{1})^{T} & (I_{m}\otimes a_{2})^{T}& \cdots & (I_{m}\otimes a_{z})^{T} \end{bmatrix}^{T}.
\end{aligned}
\end{equation}
\end{lemma}

\begin{proof}
For any $a_{\ell}\in\R^{n\times1}$ and $b_{\ell}\in\R^{1\times m}$, $\ell\in\N_{\leq z}$,  one can show $a_{\ell}b_{\ell}=(b_{\ell}\otimes I_{n})(I_{m} \otimes a_{\ell})$. Thus,
 \[
AB=\sum_{\ell=1}^{z} a_{\ell}b_{\ell}= \sum_{\ell=1}^{z} (b_{\ell}\otimes I_{n})(I_{m} \otimes a_{\ell})
\]
which can be written in the vector form $\bar{B}\bar{A}$ with $\bar{B}$ and $\bar{A}$ defined as in \eqref{eq:matrixSwitch}.
\end{proof}

We refer to $\bar{A}$ and $\bar{B}$ as lifted matrices of $A$ and $B$ since we are lifting  $A$ and $B$ to higher dimensions. 
 

 Now using Lemma~\ref{lem:matrixSwitch}, we can define lifted  matrices $\overline{F^xG_{u}}$, $\overline{F^u}$, $\overline{K^{x}}$, and $\overline{K^{u}}$, such that $F^xG_{u}K=\overline{K^{x}}~ \overline{F^xG_{u}}$ and $F^{u}K=\overline{K^{u}}~ \overline{F^{u}}$. Note after the lifting the matrix multiplications are reordered as suggested by Lemma \ref{lem:matrixSwitch}, $\overline{K^{x}}$ and $\overline{K^{u}}$ are the design variables since they are constructed from the elements of $K$. Therefore, we define the truncation mappings based on the constant  parts of these mappings as follows
\begin{equation} \label{eq:TruncationMapping_CL_General}
S(W)=\begin{pmat} [.]
S^{cl} \cr\- 
S^{ol} \cr\- 
S^{u} \cr
\end{pmat}=\begin{pmat} [.]
\overline{F^{x}G_{u}} \cr\- 
F^{x}G_{w} \cr\- 
\overline{F^{u}} \cr
\end{pmat}W.
\end{equation}
 Let $\W_{N}$ be a set of $N$ random samples satisfying the desired confidence level and risk of failure. Also let $\mathcal{S}=\conv S(\W_{N})$ where $S(W)$ is defined as \eqref{eq:TruncationMapping_CL_General}. The truncated sample set 
 $\W_{\hat{N}}$ can be computed so that $\hat{\mathcal{S}}=\conv S(\W_{\hat{N}})$ is the $\epsilon$-ACH of $\mathcal{S}$, i.e., $d_{H}(\hat{\mathcal{S}},\mathcal{S})\leq \epsilon$ where $d_{H}$ is defined in \eqref{eq:d_H_inf}. A greedy algorithm  to computed $\W_{\hat{N}}$ is later given in Section~\ref{sec:Implementation}. 
 
 Next, we will explain the buffering process to compensate for the truncation error due to using  fewer number of samples than suggested by scenario approach and  removing $K$ from the uncertainty mappings.  

\subsection{ State/Input Buffer Computation} \label{sec:ComputingBuffers}

 Given $\W_{\hat{N}}$ (as computed in Section~\ref{sec:TruncationMapping}) and the truncation mapping \eqref{eq:TruncationMapping_CL_General},   using \eqref{eq:EpsilonXF_kl} we compute
 \begin{equation} \label{eq:varepsilon}
  \varepsilon=\begin{pmat}[{||}] {\varepsilon^{cl}}^{T} & {\varepsilon^{ol}}^{T} & {\varepsilon^{u}}^{T} \cr \end{pmat}^{T}  
 \end{equation}
 and define
 \begin{subequations} \label{eq:bufferCoeff}
 	\begin{align}
\epsilon^{cl}&=d_{H}(\hat{\mathcal{S}}^{cl},\mathcal{S}^{cl})=\|\varepsilon^{cl}\|_{\infty}, \\
\epsilon^{ol}&=d_{H}(\hat{\mathcal{S}}^{ol},\mathcal{S}^{ol})=\|\varepsilon^{ol}\|_{\infty},\\
\epsilon^{u}&=d_{H}(\hat{\mathcal{S}}^{u},\mathcal{S}^{u})=\|\varepsilon^{u}\|_{\infty}.
 	\end{align}
 \end{subequations} 
Note that since $d_{H}(\hat{\mathcal{S}},\mathcal{S})=\|\varepsilon\|_{\infty}=\epsilon$, it is concluded from \eqref{eq:varepsilon} that $\epsilon=\max\{\epsilon^{cl},\epsilon^{ol},\epsilon^{u}\}$, and consequently 
 $\epsilon^{cl}, \epsilon^{ol}, \epsilon^{u}$ are upper bounded by $\epsilon$.  Also define,
\begin{equation}
	\kappa_{t}=vec(K_{1,0}, K_{2,0}, \cdots K_{t-1,t-2}) .
\end{equation}
 For $t\in \N_{\leq p}$, $\kappa_{t}$ is the vectorized form of block rows of $K$ associated with time instances up to $t$. For instance, $\kappa_{1}$ is a zero vector,   $\kappa_{2}$ is the vectorized form of $K_{1,0}$, and $\kappa_{3}$ is a vector that consists of all elements of $K_{1,0}$, $K_{2,0}$ and $K_{2,1}$.   
 We claim that by defining the following buffered constraint sets at time $t$,  the prescribed confidence level is preserved. This will be proven later in Theorem~\ref{thm:ACH_CL_General}.  
 \begin{subequations} \label{eq:DynamicBuffers}
 	\begin{align}
 	\hat{\X}_{t}&=\{x|fx\leq \textbf{1}-\textbf{1} \epsilon^{cl}\|\kappa_{t}\|-\varepsilon^{ol}_{t}\}, \\
 	\hat{\U}_{t}&=\{u|fu\leq \textbf{1}-\textbf{1} \epsilon^{u}\|\kappa_{t}\|\},
 	\end{align}
 \end{subequations}
\subsection{Approximated Problem} \label{sec:AproximatedProblem}
 
 Given the truncated sample set  (as computed in Section~\ref{sec:TruncationMapping}) and buffer coefficients calculated using \eqref{eq:bufferCoeff}, we suggest the following truncated problem. The theoretical properties of the solutions of this problem are summarized  in Theorem~\ref{thm:ACH_CL_General}.

\vspace{2mm}
\Ovalbox{
	\begin{minipage}{0.44\textwidth}
		\begin{problem} \label{prb:ACH_CL_General}  closed loop CCTO via $\epsilon$-ACH 
			\begin{align*} \label{eq:ApproximateProblem}
			\begin{aligned}
			\displaystyle\min_{K,V,\zeta} ~ &\quad \mathbb{E}(J(X,U)) \\
			\mbox{s.t. } & \hspace{1mm}
			X^{(i)}= G_{x}x_{0} + G_{u} V +(G_{u} K+ G_{w})W^{(i)} \\
			& \hspace{1mm}  F^{x}X^{(i)}\leq \textbf{1}_{pn_{c_x}} - \epsilon^{cl} \zeta^{x} -\varepsilon^{ol} \hspace{5mm}  i=1,\cdots,\hat{N} \\
			& \hspace{1mm}  F^{u}(KW^{(i)}+V)\leq \textbf{1}_{pn_{c_u}}-\epsilon^{u}\zeta^{u}   \\
			& \hspace{1mm}  \|\kappa_{t}\|\leq \zeta_{t}  \hspace{36mm}  t=1,\cdots,p
			\end{aligned} 
			\end{align*}
where $\zeta=\begin{bmatrix} \zeta_{1},\cdots,\zeta_{p} \end{bmatrix}^{T} \in \R^{p}$ and
\begin{subequations} \label{eq:Zeta_xu}
	\begin{align}
	\zeta^{x}&=\begin{bmatrix}\textbf{1}_{n_{c_x}}^{T}\otimes\zeta_{1}, \cdots,\textbf{1}_{n_{c_x}}^{T}\otimes\zeta_{p}\end{bmatrix}^{T}, \\
	\zeta^{u}&=\begin{bmatrix}\textbf{1}_{n_{c_u}}^{T}\otimes\zeta_{1}, \cdots,\textbf{1}_{n_{c_u}}^{T}\otimes\zeta_{p}\end{bmatrix}^{T},\\
	\kappa_{t}&=vec(K_{1,0}, K_{2,0}, \cdots K_{t-1,t-2}) .
	\end{align}
\end{subequations}
		\end{problem}
	\vspace{2mm}
\end{minipage}}
\vspace{3mm}

\noindent
It is important to note that Problem~\ref{prb:ACH_CL_General} is convex for any convex $J$.  $\epsilon^{cl} \in\R$, $\varepsilon^{ol} \in\R^{pn_{c_x}}$, and $\epsilon^{u} \in\R$ are computed offline using \ref{eq:bufferCoeff}, and $\zeta$ adjusts the buffers based on the optimal $K$. 

\begin{theorem} \label{thm:ACH_CL_General}
Let $\W_{N}$ be a set of $N$ random disturbances where $N$ is calculated from \eqref{eq:ScenarioSpecification} while $n_{\theta}$ indicates the number of decision variables in Problem~\ref{prb:ACH_CL_General}. Let the truncation mapping be given by \eqref{eq:TruncationMapping_CL_General} and  $\W_{\hat{N}}$ be  computed such that $d_{H}(\hat{\mathcal{S}},\mathcal{S})\leq \epsilon$, where $\hat{\mathcal{S}}=\conv S(\W_{\hat{N}})$ and $\mathcal{S}=\conv S(\W_{N})$. Let $K_{\hat{N}}^{*}$, $V_{\hat{N}}^{*}$, $\zeta^{*}_{\hat{N}}$ be the optimal solution of Problem~\ref{prb:ACH_CL_General}. 	
For the closed-loop system given by \eqref{eq:sys_dynamics} with control input  $U_{\hat{N}}^{*}=K_{\hat{N}}^{*}W+V_{\hat{N}}^{*}$ and  $W$  with the probability distribution  $\W^{p}$,  the following probabilistic guarantees hold true:
%
\begin{subequations} \label{eq:proof1_1}
\begin{align}
&\mathbb{P}(\mathbb{P}(X(U_{\hat{N}}^{*},W)\in \X^{p}) \ngtr 1-\delta)\leq \beta, \\
&\mathbb{P}(\mathbb{P}(U_{\hat{N}}^{*}\in \U^{p}) \ngtr 1-\delta)\leq \beta. 
\end{align}
\end{subequations}
\end{theorem}

\begin{proof}
We already know that \eqref{eq:proof1_1} holds for $\hat{N}=N$ for any $K$ and $V$ (including $K^{*}_{\hat{N}}$ and $V^{*}_{\hat{N}}$), which  is implied  from  Corollary~3.4 in \cite{calafiore2010random} where this bound on the number of sufficient samples is presented, given in \eqref{eq:scenarioN}, for any prescribed confidence level and failure risk.  
Note that this bound is independent of the trajectories, while it depends on the number of decision variables $n_\theta$.

Next we prove that $\textbf{1} \epsilon^{cl}\zeta_{t}+\varepsilon^{ol}_{t}$ and $\textbf{1} \epsilon^{u}\zeta_{t}$ represent guaranteed upper bounds at time instant $t$ on the state and the input violations due to truncation, respectively. Equivalently we prove that keeping the state/input truncated uncertainty envelope in the buffered state/input constraint set ensures that the original state/input uncertainty envelope remains in the original state/input constraint set.  

%

Violation on the $\ell^{th}$ state constraint at time instant $t$ can be calculated as in (\ref{eq:proof1_2}a) where $X(W)$ is defined as in \eqref{eq:sys_stacked}, and subscript $t,\ell$ denotes the row associated with time instant $t$ and constraint $\ell$. (\ref{eq:proof1_2}a)  simplifies to (\ref{eq:proof1_2}b) by using the fact that $\max_{W\in\W_{N}}\min_{\alpha\in \mathbb{C}} \|*\|=0$  for any $*$ that is independent of $W$. (\ref{eq:proof1_2}c) is resulted using triangle inequality. According to Lemma~\ref{lem:matrixSwitch} there always exist at least one  $\overline{K^{x}}$ and one $\overline{F^xG_{u}}$ such that $F^xG_{u}K$  multiplication can be reordered as $\overline{K^{x}}~ \overline{F^xG_{u}}$. Hence, (\ref{eq:proof1_2}c) equals to (\ref{eq:proof1_2}d). Note that $\overline{F^{x}G_{u}}W$ and $F^{x}G_{w}$,  in (\ref{eq:proof1_2}d), denote $S^{cl}(W)$ and $S^{ol}(W)$. Thus, according to the definition of the Hausdorff distance  \eqref{eq:dH}, and since the violation on $S_{t,\ell}^{cl}(W)$ and $S_{t,\ell}^{ol}(W)$ are bounded by $\epsilon_{cl}$ and $\varepsilon_{t,\ell}^{ol}$ (see \eqref{eq:bufferCoeff},\eqref{eq:varepsilon}),  (\ref{eq:proof1_2}e) is concluded.
 The last statement of (\ref{eq:proof1_2}e) is concluded by using  the fact that $F^xG_{u}$, $K$, and consequently their multiplication, are block lower triangular matrices where blocks are divided timewise from $1$ to $p$. Therefore, by setting the upper triangular part of $\overline{K^{x}}$ including the main diagonal to zero, $F^{x}G_{u}K=\overline{K^{x}}~\overline{F^{x}G_{u}}$ still remains valid. Thus, for any $\ell\in\N_{\leq n_{c_x}}$, $\|\overline{K}_{1,\ell}\|=0$ and for $2\leq t\leq p$ one can show
\begin{equation}
\|\overline{K}_{t,\ell}\|=\|vec(K_{1,0},\cdots,K_{t-1,t-2})]\|=\|\kappa_{t}\|.
\end{equation}
Since $\zeta_{t}$ is an upper bound on $\|\kappa_{t}\|$, $\textbf{1} \epsilon^{cl}\zeta_{t}{\color{red}+}\varepsilon^{ol}_{t}$ provides an upper bound on the state violation due to truncation, and consequently presents a proper buffer on the state constraint set.

\begin{figure*}
\begin{subequations} \label{eq:proof1_2}
\begin{align}
\max_{W\in\W_{N}}\min_{\alpha\in \mathbb{C}} &\left\|F_{t,\ell}^{x}X(W)-\sum_{i=1}^{\hat{N}}\alpha_{i}F_{t,\ell}^{x}X(W^{(i)})\right\|  \\
 &= \max_{W\in\W_{N}}\min_{\alpha\in \mathbb{C}} \left\|F_{t,\ell}^{x}(G_{u}KW+G_{w}W)-\sum_{i=1}^{\hat{N}}\alpha_{i}F_{t,\ell}^{x}(G_{u}KW^{(i)}+G_{w}W^{(i)})\right\| \\
&\leq\max_{W\in\W_{N}}\min_{\alpha\in \mathbb{C}} \left(\left\|F_{t,\ell}^{x}G_{u}KW-\sum_{i=1}^{\hat{N}}\alpha_{i}F_{t,\ell}^{x}G_{u}KW^{(i)}\right\| + \left\| F_{t,\ell}^{x}G_{w}W -\sum_{i=1}^{\hat{N}}\alpha_{i}F_{t,\ell}^{x}G_{w}W^{(i)}\right\| \right) \\
&=\max_{W\in\W_{N}}\min_{\alpha\in \mathbb{C}} \left(\left\|\overline{K^{x}_{t,\ell}} ( \overline{F^{x}G_{u}}W-\sum_{i=1}^{\hat{N}}\alpha_{i}\overline{F^{x}G_{u}}W^{(i)})\right\| + \left\| F_{t,\ell}^{x}G_{w}W -\sum_{i=1}^{\hat{N}}\alpha_{i}F_{t,\ell}^{x}G_{w}W^{(i)}\right\| \right)\\
&\leq  \|\overline{K^{x}_{t,\ell}}\| d_{H}(\hat{\mathcal{S}}^{cl},\mathcal{S}^{cl})+ d_{H}(\hat{\mathcal{S}}^{ol},\mathcal{S}^{ol}) =
\|\overline{K^{x}_{t,\ell}}\|\epsilon^{cl}+\varepsilon^{ol}_{t,\ell}=
 \|\kappa_{t}\|\epsilon^{cl}+\varepsilon^{ol}_{t,\ell}.
\end{align}
\end{subequations}	
\end{figure*}

 Since $\overline{K^{u}}$ has the same structure as $\overline{K^{x}}$ (with different number of rows), one can similarly prove $\epsilon^{u}\zeta_{t}$ provides an upper bound on the input constraint violation due to truncation at time $t$. Hence, the proof is completed.
 	 
\end{proof}

\begin{remark}
	Problem~\ref{prb:ACH_CL_General} has $p$ decision variables more than Problem~\ref{prb:Scenario} due to the addition of  the slack variable $\zeta$. Therefore, more original samples are required for constructing Problem~\ref{prb:ACH_CL_General}. However, these samples are only used to define the buffers and only $\hat{N}$ of them are directly used in the online computation of  the closed-loop input.  
\end{remark}


\section{Implementation} \label{sec:Implementation}

The implementation of the proposed truncated CCTO problem with disturbance feedback is simply executed in two offline and online steps. 

\textbf{offline:} Using \eqref{eq:scenarioN} one can find the required number of samples based on the desired confidence level $\delta$ and risk of failure $\beta$ for Problem~\ref{prb:ACH_CL_General} with variables $K$, $V$, and $\zeta$ (typically another slack variable may be also used to convert the quadratic cost function to a linear one) and generate $\W_{N}$ through randomly sampling the disturbance set $\W^{p}$. Given $F^{x}$, $F^{u}$, and $G^{u}$, one can use Lemma~\eqref{lem:matrixSwitch} to construct $\overline{F^{x}G_{u}}$ and $\overline{F^u}$ and then map $\W_{N}$ to $S(\W_{N})$ using \eqref{eq:TruncationMapping_CL_General}. It is notated that $S(\W_{N})$ is a set of $N$ vectors corresponding to $N$ different scenarios(disturbances). To select the dominant $\hat{N}$ scenarios, the following greedy algorithm can be used.

\begin{enumerate}
	\item  Initiate  $\W_{\hat{N}}$ with the argument of an extreme point of $S(\W_{N})$ (with some abuse of notation, say $W_{1}$). For instance, the farthest element from any element in a set is an extreme point. Find $\varepsilon$ from \eqref{eq:EpsilonXF_kl} by setting $\W_{\hat{N}}=\W_{1}=W_{1}$. 
	\item At step $\ell$ given $\W_{\ell-1}$ and $\varepsilon$ find $W_{\ell}$ so that $d_{H}(\conv S(\W_{\ell-1}\cup W_{\ell}),\conv S(\W_{N}))$ is minimized. According to the Hausdorff distance definition given in \eqref{eq:EpsilonXF_kl}-\eqref{eq:d_H_inf} this is simply to select the $W_{\ell} \in \W_{N} \backslash\W_{\ell-1}$ that minimizes the largest element of $\varepsilon$. This procedure is fast since only a search over the elements of a vector is required!
	\item Repeat step 2 while $\ell\leq\hat{N}$ and $\varepsilon\neq \textbf{0}$.
\end{enumerate}

After finding the best $\hat{N}$ samples, $\varepsilon^{ol}$, $\varepsilon^{cl}$, $\varepsilon^{u}$ and consequently $\epsilon^{cl} \in\R$ and $\epsilon^{u} \in \R$ can be calculated offline from $\varepsilon$. 

\textbf{online:} Given $\W_{\hat{N}}$, $\epsilon^{cl}$, $\epsilon^{u}$ and $\varepsilon^{ol}$, the given convex problem (Problem~\ref{prb:ACH_CL_General}) is solved for $K^{*}_{\hat{N}}$, $V^{*}_{\hat{N}}$, and $\zeta^{*}_{\hat{N}}$. Note that for any  disturbance sequence $W\in\W^{p}$, $U^{*}_{\hat{N}}=K^{*}_{\hat{N}}W+V^{*}_{\hat{N}}$ guarantees the state and input constraints satisfaction by confidence level of $1-\delta$ and failure risk of $\beta$.


 \section{Illustrative Example} \label{sec: Illustrative Example }

 \begin{figure}
 	\vspace{-0mm}
 	\centering
 	\includegraphics[width=3in]{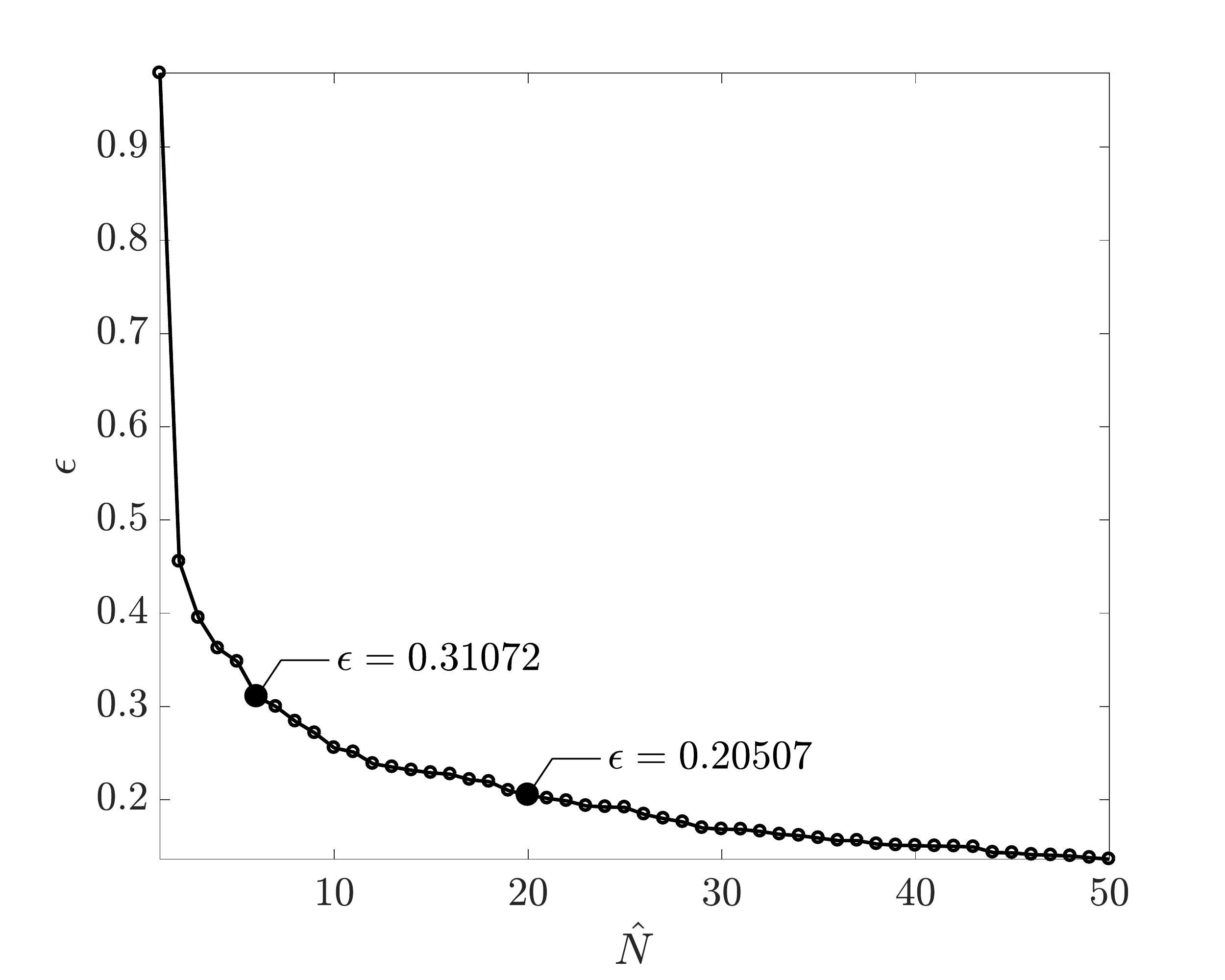}\\
 	\vspace{-3.5mm}
 	\caption{Approximation (truncation) error vs number of selected elements.}
 	\vspace{-6mm}
 	\label{fig:Epsilon}
 \end{figure}

  \begin{figure*} [t!]
	\vspace{-2mm}
	\centering
	\begin{minipage}{2.55in}
		\centering
		\includegraphics[width=2.55in]{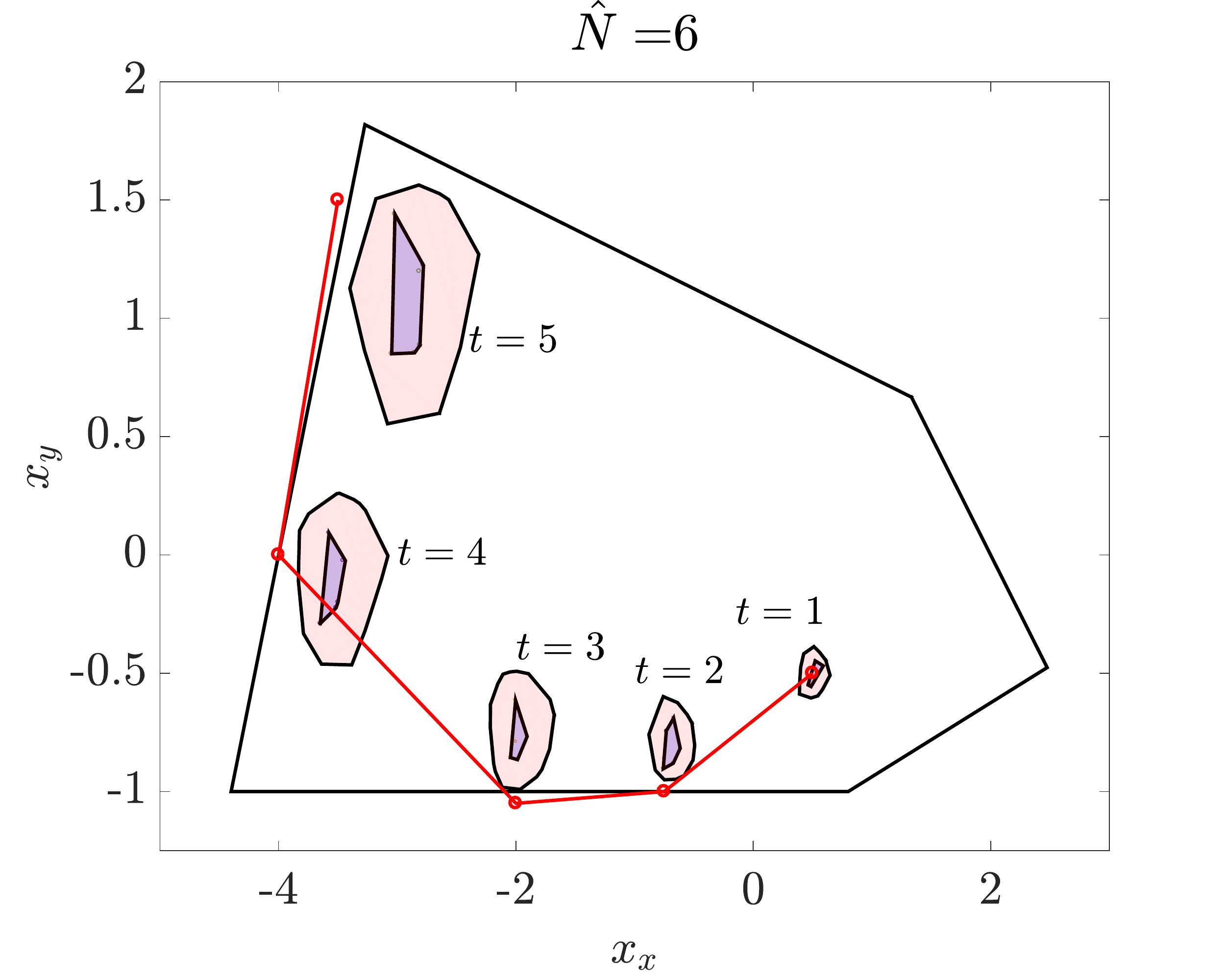}\\
	\end{minipage}
	\begin{minipage}{2.55in}
		\centering
		\includegraphics[width=2.55in]{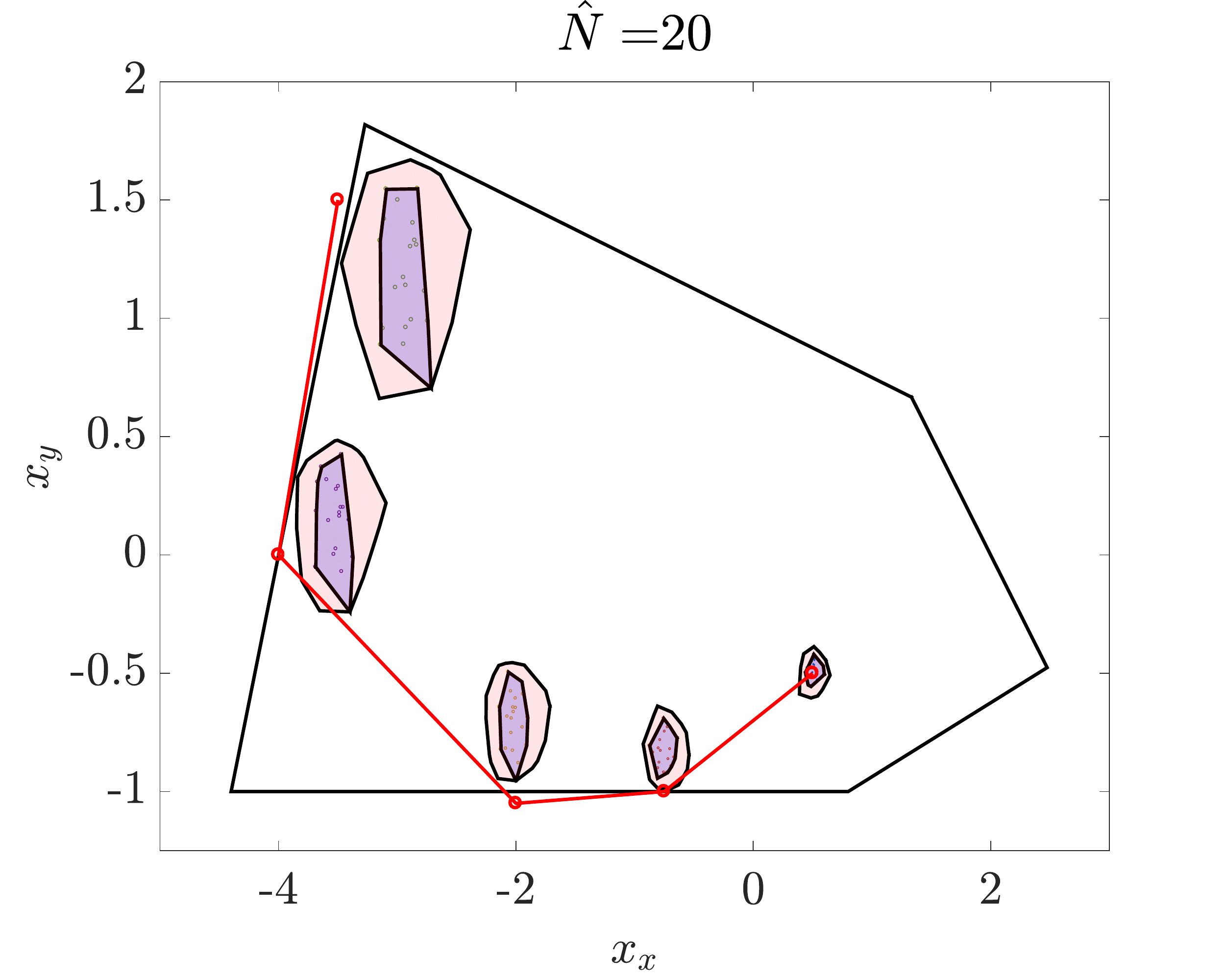}\\
	\end{minipage}
	\vspace{-2.55mm}
	\caption{state trajectory tracking results with $6$ and $20$ samples (out of $5564$ samples) while states are subject to remain in the safe region. The original and truncated uncertainty envelopes have been distinguished by bright and dark polytopes for $t=1,\cdots,5$.}
	\vspace{-5.5mm}
	\label{fig:Control Results}
\end{figure*}

 To show the application of the suggested method for closed-loop CCTO, we simulate the trajectory  performance of the following double-integrator robot (in a $2$-D plane). In this example we assume $T_s=1$ s,
 
 \begin{equation} \label{eq:SystemMatrices_RR}
 \begin{bmatrix} \dot{x}_{x}\\\dot{v}_{x} \\ \dot{x}_{y}\\\dot{v}_{y} \end{bmatrix} =\begin{bmatrix} 1&T_s &0 &0\\0 &1 &0 &0 \\ 0 &0 &1 &T_s \\0 &0 &0 &1 \end{bmatrix} \begin{bmatrix} x_{x}\\v_{x} \\ x_{y}\\v_{y} \end{bmatrix}+
 \begin{bmatrix} 0 &0\\T_s & 0\\ 0 &0 \\ 0 & T_s \end{bmatrix}\begin{bmatrix}
 u_x \\ u_y \end{bmatrix}.
  \end{equation}
  We assume $B_w=I_{4}$ and the states are disturbed by $w\sim \mathcal{N} \left(\textbf{0},\sigma_{w}^{2}\right)$ where $\sigma_{w}^{2}=diag(1e-3,4e-4,1e-3,4e-4)$. It is assumed that the uncertainty distribution is not known but we can sample the uncertainty. In addition, $x_0=\begin{bmatrix}  0.5 &0 &-.5 &0  \end{bmatrix}^{T}$. $K$ is limited to be in the convex hull of a prescribed set of  $K$ matrices and also the states and inputs are to remain into the following safe regions,
  \begin{equation}
  \begin{bmatrix}
  0.5 &0.5\\0.25 &1\\ -0.25 &0.1\\0.25 &-0.8\\ 0 & -1
  \end{bmatrix} \begin{bmatrix}
  x_{x}\\x_{y}
  \end{bmatrix}\leq \textbf{1}; 
    \begin{bmatrix}
  0.06 &0.08\\0.05 &-.15\\ 0.05 &0.08\\0 &0.2
  \end{bmatrix} u\leq \textbf{1}.
  \end{equation}
    We take $5564$ samples to ensure $\mathbb{P}(\mathbb{P}(X\notin\X)>0.02)\leq0.0001$, $\mathbb{P}(\mathbb{P}(U\notin\U)>0.02)\leq0.0001$. Figure~\ref{fig:Epsilon} shows the truncation results for $\hat{N}\in\N_{\leq 50}$. The truncation errors for $\hat{N}=6$ and  $\hat{N}=20$ have been highlighted while the control results for the two selected $\hat{N}$ is presented in Figure~\ref{fig:Control Results}. For $\hat{N}=6$ we computed $\epsilon_{u}=0.3107$, $\epsilon_{cl}=0.0155$ and for  $\hat{N}=20$, $\epsilon_{u}=0.2051$, $\epsilon_{cl}=0.0103$.  
  Figure~\ref{fig:Control Results} shows that the robot could successfully remain in the safe region while the case with $\hat{N}=6$ is slightly more conservative in this example. Inputs also remain in the desired safe regions although only the state trajectory has been plotted. Note that we have already assumed that $N$ samples are enough to meet the desired specification, and compared the scenario and truncated problems results.


\section{Conclusion} \label{sec:Conclusion}

In this paper, we presented a method to reduce the number of samples in scenario approach for a closed-loop chance constrained control  problem with disturbance feedback. In the proposed method, the most dominant samples are selected by the user and the rest are truncated. The truncation error is later compensated for by adjusting the constraint set using proper dynamic buffers. We showed that the new problem with the truncated sample set and constraints, is convex and its solution satisfies the problem specifications (the required confidence level and risk of failure). The proposed method was successfully implemented through simulations on an illustrative  $2$-D robot example.


\section*{Acknowledgements} 
\vspace{-1mm}
This work was supported by Air Force Research Laboratory grant FA8650-15-C-2546 and the Office of Naval Research (ONR) Grant No. N00014-15-IP-00052.

\bibliographystyle{IEEEtran}
\bibliography{SMPC}

\begin{thebibliography}{10}
\providecommand{\url}[1]{#1}
\csname url@samestyle\endcsname
\providecommand{\newblock}{\relax}
\providecommand{\bibinfo}[2]{#2}
\providecommand{\BIBentrySTDinterwordspacing}{\spaceskip=0pt\relax}
\providecommand{\BIBentryALTinterwordstretchfactor}{4}
\providecommand{\BIBentryALTinterwordspacing}{\spaceskip=\fontdimen2\font plus
\BIBentryALTinterwordstretchfactor\fontdimen3\font minus
  \fontdimen4\font\relax}
\providecommand{\BIBforeignlanguage}[2]{{%
\expandafter\ifx\csname l@#1\endcsname\relax
\typeout{** WARNING: IEEEtran.bst: No hyphenation pattern has been}%
\typeout{** loaded for the language `#1'. Using the pattern for}%
\typeout{** the default language instead.}%
\else
\language=\csname l@#1\endcsname
\fi
#2}}
\providecommand{\BIBdecl}{\relax}
\BIBdecl

\bibitem{pontryagin1987mathematical}
L.~S. Pontryagin, \emph{Mathematical theory of optimal processes}.\hskip 1em
  plus 0.5em minus 0.4em\relax {CRC} Press, 1987.

\bibitem{campo1987robust}
P.~J. Campo and M.~Morari, ``Robust model predictive control,'' in \emph{1987
  American Control Conference}, June 1987, pp. 1021--1026.

\bibitem{ben1998robust}
A.~Ben-Tal and A.~Nemirovski, ``Robust convex optimization,'' \emph{Mathematics
  of operations research}, vol.~23, no.~4, pp. 769--805, 1998.

\bibitem{bemporad1999robust}
A.~Bemporad and M.~Morari, ``Robust model predictive control: A survey,''
  \emph{Robustness in identification and control}, pp. 207--226, 1999.

\bibitem{charnes1959chance}
A.~Charnes and W.~W. Cooper, ``Chance-constrained programming,''
  \emph{Management science}, vol.~6, no.~1, pp. 73--79, 1959.

\bibitem{schwarm1999chance}
A.~T. Schwarm and M.~Nikolaou, ``Chance-constrained model predictive control,''
  \emph{AIChE Journal}, vol.~45, no.~8, pp. 1743--1752, 1999.

\bibitem{li2000robust}
P.~Li, M.~Wendt, and G.~Wozny, ``Robust model predictive control under chance
  constraints,'' \emph{Computers \& Chemical Engineering}, vol.~24, no. 2-7,
  pp. 829--834, 2000.

\bibitem{li2002probabilistically}
------, ``A probabilistically constrained model predictive controller,''
  \emph{Automatica}, vol.~38, no.~7, pp. 1171--1176, 2002.

\bibitem{mesbah2016stochastic}
A.~Mesbah, ``Stochastic model predictive control: An overview and perspectives
  for future research,'' \emph{IEEE Control Systems}, vol.~36, no.~6, pp.
  30--44, Dec 2016.

\bibitem{nemirovski2006convex}
A.~Nemirovski and A.~Shapiro, ``Convex approximations of chance constrained
  programs,'' \emph{SIAM Journal on Optimization}, vol.~17, no.~4, pp.
  969--996, 2006.

\bibitem{blackmore2009convex}
L.~Blackmore and M.~Ono, ``Convex chance constrained predictive control without
  sampling,'' in \emph{Proceedings of the AIAA Guidance, Navigation and Control
  Conference}, 2009, pp. 7--21.

\bibitem{ben2000robust}
A.~Ben-Tal and A.~Nemirovski, ``Robust solutions of linear programming problems
  contaminated with uncertain data,'' \emph{Mathematical programming}, vol.~88,
  no.~3, pp. 411--424, 2000.

\bibitem{nemirovski2003tractable}
A.~Nemirovski, ``On tractable approximations of randomly perturbed convex
  constraints,'' in \emph{Decision and Control, 2003. Proceedings. 42nd IEEE
  Conference on}, vol.~3.\hskip 1em plus 0.5em minus 0.4em\relax IEEE, 2003,
  pp. 2419--2422.

\bibitem{bertsimas2004price}
D.~Bertsimas and M.~Sim, ``The price of robustness,'' \emph{Operations
  research}, vol.~52, no.~1, pp. 35--53, 2004.

\bibitem{rajendran2019stochastic}
S.~Rajendran, A.~Ansaripour, M.~Kris~Srinivasan, and M.~J. Chandra,
  ``Stochastic goal programming approach to determine the side effects to be
  labeled on pharmaceutical drugs,'' \emph{IISE Transactions on Healthcare
  Systems Engineering}, vol.~9, no.~1, pp. 83--94, 2019.

\bibitem{ansaripour2016some}
A.~Ansaripour, A.~Mata, S.~Nourazari, H.~Kumin \emph{et~al.}, ``Some explicit
  results for the distribution problem of stochastic linear programming,''
  \emph{Open Journal of Optimization}, vol.~5, no.~04, p. 140, 2016.

\bibitem{calafiore2006scenario}
G.~C. Calafiore and M.~C. Campi, ``The scenario approach to robust control
  design,'' \emph{IEEE Transactions on Automatic Control}, vol.~51, no.~5, pp.
  742--753, May 2006.

\bibitem{calafiore2013stochastic}
G.~C. Calafiore and L.~Fagiano, ``Stochastic model predictive control of {LPV}
  systems via scenario optimization,'' \emph{Automatica}, vol.~49, no.~6, pp.
  1861--1866, 2013.

\bibitem{calafiore2010random}
G.~C. Calafiore, ``Random convex programs,'' \emph{SIAM Journal on
  Optimization}, vol.~20, no.~6, pp. 3427--3464, 2010.

\bibitem{lorenzen2017stochastic}
M.~Lorenzen, F.~Dabbene, R.~Tempo, and F.~Allg{\"o}wer, ``Stochastic {MPC} with
  offline uncertainty sampling,'' \emph{Automatica}, vol.~81, pp. 176--183,
  2017.

\bibitem{zhang2015convex}
X.~Zhang, A.~Georghiou, and J.~Lygeros, ``Convex approximation of
  chance-constrained {MPC} through piecewise affine policies using randomized
  and robust optimization,'' in \emph{2015 54th IEEE Conference on Decision and
  Control (CDC)}, Dec 2015, pp. 3038--3043.

\bibitem{zhang2015onthesamplesize}
X.~Zhang, S.~Grammatico, G.~Schildbach, P.~Goulart, and J.~Lygeros, ``On the
  sample size of random convex programs with structured dependence on the
  uncertainty,'' \emph{arXiv:1502.00803v2}, 2015.

\bibitem{campi2011sampling}
M.~C. Campi and S.~Garatti, ``A sampling-and-discarding approach to
  chance-constrained optimization: feasibility and optimality,'' \emph{Journal
  of Optimization Theory and Applications}, vol. 148, no.~2, pp. 257--280,
  2011.

\bibitem{sartipizadeh2016uncertainty}
H.~Sartipizadeh and T.~L. Vincent, ``Uncertainty characterization for robust
  {MPC} using an approximate convex hull method,'' in \emph{2016 American
  Control Conference (ACC)}, July 2016, pp. 2699--2704.

\bibitem{satipizadeh2018Automatica}
------, ``A new robust mpc using an approximate convex hull,''
  \emph{Automatica}, 2018.

\bibitem{sartipizadeh2018CCTO_OL}
H.~Sartipizadeh and B.~A{\c{c}}{\i}kme{\c{s}}e, ``Approximate convex hull based
  sample truncation for scenario approach to chance constrained trajectory
  optimization,'' in \emph{2018 American Control Conference (ACC)}, 2018.

\bibitem{van2002conic}
D.~Van~Hessem and O.~Bosgra, ``A conic reformulation of model predictive
  control including bounded and stochastic disturbances under state and input
  constraints,'' in \emph{Decision and Control, 2002, Proceedings of the 41st
  IEEE Conference on}, vol.~4.\hskip 1em plus 0.5em minus 0.4em\relax IEEE,
  2002, pp. 4643--4648.

\bibitem{lofberg2003approximations}
J.~Lofberg, ``Approximations of closed-loop minimax mpc,'' in \emph{Decision
  and Control, 2003. Proceedings. 42nd IEEE Conference on}, vol.~2.\hskip 1em
  plus 0.5em minus 0.4em\relax IEEE, 2003, pp. 1438--1442.

\bibitem{goulart2006optimization}
P.~J. Goulart, E.~C. Kerrigan, and J.~M. Maciejowski, ``Optimization over state
  feedback policies for robust control with constraints,'' \emph{Automatica},
  vol.~42, no.~4, pp. 523--533, 2006.

\end{thebibliography}

\end{document}